\documentclass{amsart}
\usepackage{latexsym}
\usepackage{amssymb}
\usepackage[pdftex]{hyperref}

\newtheorem{theorem}{Theorem}[section]
\newtheorem{lemma}[theorem]{Lemma}

\newtheorem*{conjectureB}{Conjecture B}
\newtheorem*{theoremA}{Theorem A}
\newtheorem*{theoremC}{Theorem C}
\newtheorem*{theoremD}{Theorem D}

\theoremstyle{definition}

\newtheorem*{notation}{Notation}

\theoremstyle{remark}

\numberwithin{equation}{section}

\newcommand{\Aut}{{\mathrm {Aut}}}
\newcommand{\Out}{{\mathrm {Out}}}

\newcommand{\Soc}{{\mathrm {Soc}}}

\newcommand{\cl}{{\mathrm {cl}}}

\newcommand{\NN}{{\mathbb N}}

\newcommand{\tb}{\hspace{0.5mm}^{3}\hspace*{-0.2mm}}

\begin{document}
\title[Multiplicities of conjugacy class sizes of finite groups]
{Multiplicities of conjugacy class sizes\\ of finite groups}

\author{Hung Ngoc Nguyen}
\address{Department of Theoretical and Applied Mathematics, The University of Akron, Akron,
Ohio 44253} \email{hungnguyen@uakron.edu}



\date{\today}

\begin{abstract} It has been proved recently by Moret\'{o}
\cite{Moreto} and Craven \cite{Craven} that the order of a finite
group is bounded in terms of the largest multiplicity of its
irreducible character degrees. A conjugacy class version of this
result was proved for solvable groups by Zaikin-Zapirain
\cite{Jaikin2}. In this note, we prove that if $G$ is a finite
simple group then the order of $G$, denoted by $|G|$, is bounded in
terms of the largest multiplicity of its conjugacy class sizes and
that if the largest multiplicity of conjugacy class sizes of any
quotient of a finite group $G$ is $m$, then $|G|$ is bounded in
terms of $m$.
\end{abstract}

\maketitle


\section{Introduction}

Let $G$ be a finite group. The multiplicity of a (complex) character
degree of $G$ is the number of distinct irreducible characters of
that degree of the group $G$. Moret\'{o} conjectures in
\cite{Moreto} that if the largest multiplicity of a character degree
of $G$ is $m$, then the order of $G$ is $m$-bounded. Here, we say
that $|G|$ is \emph{$m$-bounded} or \emph{bounded in terms of $m$}
if $|G|<f(m)$ for some real-valued function $f$ on $\NN$. In the
same paper, Moret\'{o}, by using the classification of finite simple
groups, proves that the conjecture is true for all finite groups if
it is true for the symmetric groups. The problem for the symmetric
groups is naturally combinatoric and was done by Craven in
\cite{Craven}.

\begin{theoremA}[Craven \cite{Craven}, Moret\'{o} \cite{Moreto}]\label{Moreto} If the largest multiplicity of
irreducible ordinary character degrees of a finite group $G$ is $m$,
then $|G|$ is $m$-bounded.
\end{theoremA}

In a similar way, we define the multiplicity of a conjugacy class
size of $G$ to be the number of distinct conjugacy classes of that
size. Because of the natural duality between the characters and
conjugacy classes, one might expect a conjugacy class version of
Theorem~A:

\begin{conjectureB} If the largest multiplicity of conjugacy class
sizes of a finite group $G$ is $m$, then $|G|$ is $m$-bounded.
\end{conjectureB}

\noindent We would like to point out that this problem is not new,
at least for solvable groups. Indeed, it has been proved by
Jaikin-Zapirain for nilpotent groups in \cite{Jaikin1} and then for
solvable groups in \cite{Jaikin2}.

Our first result of this note is the proof of Conjecture B for
finite simple groups.

\begin{theoremC} Let $G$ be a finite simple group. If the largest multiplicity of conjugacy class
sizes of $G$ is $m$, then $|G|$ is $m$-bounded.
\end{theoremC}

Using Theorem C and drawing upon some techniques of Moret\'{o} in
\cite{Moreto}, we prove a weaker version of Conjecture B.

\begin{theoremD}\label{D} If the largest multiplicity of conjugacy class
sizes of any quotient of a finite group $G$ is $m$, then $|G|$ is
$m$-bounded.
\end{theoremD}

\begin{notation} Let $X$ be a finite group and $x\in X$. We denote
the centralizer of $x$ in $X$ by $C_X(x)$, the conjugacy class of
$x$ in $X$ by $\cl_X(x)$, and the size of this class by
$|\cl_X(x)|$. Other notations are standard.
\end{notation}


\section{Proof of Theorem C}

The aim of this section is to prove Theorem C. The case of cyclic
groups of prime orders is obvious. Therefore, it remains to consider
finite simple groups of Lie type and alternating groups.

\begin{lemma}\label{lietype} Theorem \textup{C} holds for finite simple groups of Lie type.
\end{lemma}

\begin{proof} Let $G_n(q)$ be a simple group of Lie type (twisted or untwisted) of rank $n$ defined over a field of $q$ elements and let
$k^\ast(G_n(q))$ denote the number of orbits of $\Aut(G_n(q))$
acting on $G_n(q)$. Lemma~6.3 of \cite{Moreto} claims that
$$\frac{k^\ast(G_n(q))}{d(|\Aut(G_n(q))|)}\rightarrow \infty \text{ as } q\rightarrow
\infty,$$ where $d(|\Aut(G_n(q))|)$ is the number of positive
divisors of $|\Aut(G_n(q))|$. Since every conjugacy class size of
$G_n(q)$ divides $|G_n(q)|$, it follows that the average
multiplicity of conjugacy class sizes of $G_n(q)$ tends to $\infty$
as $q$ tends to $\infty$ and therefore we are done for the case of
exceptional simple groups, where the rank $n$ is $\leq 8$.

Assume that $G_n(q)$ is a simple classical group. By a result of
Malle (see Corollary~3.4 of~\cite{Moreto}), $G_n(q)$ has a
semisimple element $s$ of order at least $(q^{[n+1/2]}-1)/(n+1)$
conjugate to at most $2n+1$ of its powers (actually, for the
orthogonal groups, Malle states only the results for projective
special orthogonal groups but the same arguments work for the simple
groups). Note that $s$ and $s^t$ have the same centralizers whenever
$t$ is coprime to $|s|$. It follows that $G_n(q)$ has at least
$$\frac{\min_{k\geq (q^{[n+1/2]}-1)/(n+1)}\varphi(k)}{2n+1}$$ conjugacy
classes (of semisimple elements) of the same sizes, where $\varphi$
is the Euler's totient function. It is well-known that
$\varphi(k)>\sqrt{k/2}$ for any $k>0$. Therefore, the above quotient
tends to $\infty$ as $|G_n(q)|$ tends to $\infty$, as desired.
\end{proof}

In the following lemma, $S_n$ and $A_n$ are the symmetric and
alternating groups, respectively, of degree $n$.

\begin{lemma}\label{alternating} Theorem \textup{C} holds for the alternating groups.
\end{lemma}

\begin{proof} It is known that conjugacy classes of the symmetric group $S_n$ of degree
$n$ is in one-to-one correspondence with partitions of size $n$. Let
us recall some standard notation and terminology of partitions. A
partition $\lambda$ of size $n$ is a finite sequence
$(\lambda_1,\lambda_2,...)$ such that
$\lambda_1\geq\lambda_2\geq\cdot\cdot\cdot$ and
$\lambda_1+\lambda_2+\cdot\cdot\cdot=n$. Each $\lambda_i$ is called
a part of $\lambda$. Let $m_i(\lambda)$ denote the number of parts
of $\lambda$ of size $i$. We say that $\lambda$ is even (odd, resp.)
if permutations of cycle type $\lambda$ are even (odd, resp.). We
will often call the centralizer size of a permutation (in $S_n$) of
cycle type $\lambda$ by centralizer size of $\lambda$ and denote it
by $C(\lambda)$. It is well-known that
$$C(\lambda)=\prod_{i}(i^{m_i(\lambda)})(m_i(\lambda))!.$$
In particular, if all the parts of $\lambda$ are different then
$C(\lambda)$ is product of those parts.

We will prove that for any $M$, there exists an integer $N(M)$ big
enough so that for each $n\geq N$ there are $M$ even partitions of
size $n$ with the same centralizer size. This will imply the lemma
since the centralizer size of an even permutation in $A_n$ is either
equal or a half of that in $S_n$.

If $\lambda=(\lambda_1,\lambda_2,...)$ is a partition and
$\lambda_0\geq\lambda_1$, we define the partition
$$(\lambda_0,\lambda):=(\lambda_0,\lambda_1,\lambda_2,...).$$
It is clear that
\begin{equation}\label{centralizer}
C((\lambda_0,\lambda))=\lambda_0C(\lambda) \text{ if }
\lambda_0>\lambda_1.\end{equation}

For each $i\in \NN$, let $(a_i,b_i,c_i)$ be either $(10,9,1)$ or
$(15,3,2)$. Then $(a_i,b_i,c_i)$ is a partition of size $20$ with
centralizer size $90$. Choose an odd integer $k$ so that $2^{k}>M$.
Consider the set $P$ of $2^k$ odd partitions of the form
$$(21^{k-1}a_{k},
21^{k-1}b_{k},21^{k-1}c_{k},...,21a_2,21b_2,21c_2,a_1,b_1,c_1).$$
All partitions in $P$ have same size
$20(21^{k-1}+21^{k-2}+\cdot\cdot\cdot+1)=21^k-1$ and centralizer
size $90^k\cdot 21^{3k(k-1)/2}$. Therefore, if $n$ is an even
integer bigger than $21^k+21^{k-1}15$, we obtain $2^k$ even
partitions of the form $(n-|\lambda|,\lambda)$ with $\lambda\in P$,
all of size $n$ and centralizer size $(n-|\lambda|)C(\lambda)$
by~\eqref{centralizer}.

To handle odd $n$, we consider the set $P'$ of $2^{k+1}$ even
partitions of the form
$$(21^{k}a_{k+1},
21^{k}b_{k+1},21^{k}c_{k+1},...,21a_2,21b_2,21c_2,a_1,b_1,c_1).$$ As
before, all partitions in $P'$ have the same size $21^{k+1}-1$ and
centralizer size $90^{k+1}\cdot 21^{3k(k+1)/2}$. Hence, if $n$ is an
odd integer bigger than $21^{k+1}+21^k15$, we obtain $2^{k+1}$ even
partitions of the form $(n-|\lambda|,\lambda)$ with $\lambda\in P'$.

The lemma is completely proved.
\end{proof}


\section{Proof of Theorem D}

We collect some results needed in the proof of Theorem D. The
following number-theoretic result is in \cite{Erdos} and is stated
as Lemma~4.1 in~\cite{Moreto}.

\begin{lemma}\label{factorial} \textup{(i)} $\lim_{n\rightarrow \infty}\frac{d(n!)}{2^{\frac{c\log n!}{(\log\log
n!)^2}}}$, where $c$ is a constant and $d(n!)$ is the number of
divisors of $n!$.

\textup{(ii)}
$\lim_{k\rightarrow\infty}\frac{a^{ck}}{(k+1)d(k!)}=\infty$, where
$c$ is a positive constant. \hfill$\Box$
\end{lemma}

The next result is due to Babai and Pyber \cite{Babai-Pyber} and
appears as Theorem~4.3 in~\cite{Moreto}.

\begin{lemma}\label{orbit} Let $G$ be a permutation group on a set $\Omega$ of
cardinality $k$ and assume that $G$ does not contain any alternating
group larger than $A_n$ as a composition factor. Then the number of
orbits of $G$ on the power set $\mathcal{P}(\Omega)$ is at least
$a^{k/n}$ where $a$ is some constant.\hfill$\Box$
\end{lemma}

As mentioned earlier, Conjecture B has been proved for solvable
groups by Jaikin-Zapirain~\cite{Jaikin2}. We need this result in the
proof of Theorem~D.

\begin{theorem}[Zaikin-Zapirain \cite{Jaikin2}]\label{Jaikin-solvable} If the largest multiplicity of conjugacy class
sizes of a finite solvable group $G$ is $m$, then the order of $G$
is $m$-bounded.\hfill$\Box$
\end{theorem}

The following lemma is a nice property of the action of the
automorphism group of a simple group on its conjugacy classes which
might be of independent interest.

\begin{lemma}\label{simple-extension} Let $S$ be a finite non-abelian simple group.
Then there is a nontrivial conjugacy class of $S$ invariant under
the conjugate action of $\Aut(S)$.
\end{lemma}

\begin{proof} Checking the Atlas \cite{Atl1} case by case, we see
that every conjugacy class of involutions in a simple sporadic group
or the Tits group is invariant under the automorphism group. It is
easy to see that the same conclusion also holds for alternating
groups. So we can assume that $S$ is a finite simple group of Lie
type.

\medskip

1) First we prove the lemma for Suzuki and Ree groups. In this case,
$\Out(S)$ is a cyclic group of field automorphisms of $S$.
Therefore, the action of $\Out(S)$ on conjugacy classes is
permutation isomorphic to its action on irreducible characters of
$S$. Since the \emph{Steinberg character} of $S$ is invariant under
$\Out(S)$, there must be a nontrivial conjugacy class of $S$
invariant under $\Out(S)$.

\medskip

2) Next we assume that $S$ is $F_4(2^n)=F_4(q)$. The conjugacy
classes of $S$ is described in \cite{Shinoda}. Inspecting the list
of class representatives and the orders of their centralizers in
\cite[Theorem~2.1]{Shinoda}, we see that $S$ has a unique conjugacy
classes of involutions with order of its representative centralizer
$q^{24}(q^2-1)(q^4-1)$. This class is therefore invariant under
$\Aut(S)$.

\medskip

3) Now we handle the case where $S$ is a finite simple group of Lie
type in even characteristic but not the ones considered in 1) and
2). Note that $B_n(2^m\simeq C_n(2^m)$, we can assume that $S$ is
not of type $B_n$ in even characteristic. As claimed in~\cite[page
103]{Gorenstein}, the class of \emph{long root subgroups} of $S$ is
invariant under $\Aut(S)$ and all \emph{long root elements}, which
are nonidentity elements of long root subgroups, are $S$-conjugate.
This means that the unique conjugacy class of long root elements of
$S$ is invariant under $\Aut(S)$.

\medskip

4) Finally, we assume that $S$ is a finite group of Lie type in odd
characteristic. Table~4.5.1 of~\cite{Gorenstein} describes all
conjugacy classes of involutions of $S$ including the structures of
centralizers of class representatives. In general, these
centralizers is ``close'' (cf. the interpretation of
\cite[Table~4.5.1]{Gorenstein}) to a central product of finite
groups of Lie type. For each type of $S$, we choose a conjugacy
class so that the structure of corresponding centralizer is
different from that of other centralizers. This class is therefore
invariant under $\Aut(S)$. The conjugacy class with representative
$t_1$ or $t_1'$ will work in all cases except $S=\tb D_4(q)$ where
we choose the class with representative $t_2$.
\end{proof}

\begin{proof}[Proof of Theorem~\textup{D}] The proof is divided into the following
steps:

\medskip

1) If $A_n$ is a composition factor of $G$, then $n$ is $m$-bounded.

\noindent\emph{Proof.} Suppose that $N/M$ is a chief factor of $G$
that is a direct product of copies of $A_n$. Then $N/M$ is also a
chief factor of $G/M$. Since $G/M$ satisfies the hypothesis of the
theorem, we can assume that $M=1$. The subgroup $N$ is then a
minimal normal subgroup of $G$ and we assume that $N$ is the direct
product of $k$ copies of $A_n$. The quotient $H=G/C_G(N)$ is
embedded in $\Aut(N)=\Aut(S)\wr S_k$. Note that we can consider $N$
as a subgroup of $H$.

Assume the contrary that $n$ is not $m$-bounded.
Lemma~\ref{alternating} then implies that $A_n$ can have arbitrarily
many conjugacy classes of the same size. Let's pick $n$ big enough
so that $A_n$ has $3m+1$ conjugacy classes $K_1, K_2,...,K_{3m+1}$
all of size $s$. We then obtain $3m+1$ corresponding conjugacy
classes of $N$ of size $s^k$: $K'_i=K_i\times
K_i\times\cdot\cdot\cdot\times K_i$, $i=1,2,...,3m+1$. Since the
orbit of $K'_i$ under the action of $\Aut(N)=\Aut(S)\wr S_k$
consists of $1$ or $2$ conjugacy classes (of size $s^k$), the same
thing happens to $H$ since $H\leq \Aut(N)$. If there are $m+1$ of
$K'_i$s fixed under $H$, then $H$ has $m+1$ conjugacy classes of
size $s^k$, which is a contradiction. If there are at most $m$ of
$K'_i$s fixed under $H$, then we would have at least $2m+1$ of
$K'_i$s whose orbits under $H$ all have two conjugacy classes of
size $s^k$. It follows that $H$ has at least $m+1$ conjugacy classes
of size $2s^k$, a contradiction again.

\medskip

2) Let $O_\infty(G)$ denote the maximal normal solvable subgroup of
$G$. We claim that if $|G:O_\infty(G)|$ is $m$-bounded then $|G|$ is
$m$-bounded. Let $x$ be an arbitrary element in $O_\infty(G)$. Then
$$\frac{|\cl_G(x)|}{|\cl_{O_\infty(G)}(x)|}=\frac{|G|}{|O_\infty(G)|}\cdot\frac{|C_{O_\infty(G)}(x)|}{|C_G(x)|},$$
which divides $|G:O_\infty(G)|$. Since multiplicity of every
conjugacy class size of $G$ is less than or equal to $m$ and
$|G:O_\infty(G)$ is $m$-bounded, the multiplicity of any conjugacy
class size of $O_\infty(G)$ is $m$-bounded. Using
Theorem~\ref{Jaikin-solvable}, we see that $|O_\infty(G)|$ is
$m$-bounded and therefore $|G|=|O_\infty(G)||G:O_\infty(G)|$ is
$m$-bounded.

\medskip

3) Now we show that $|G:O_\infty(G)|$ is $m$-bounded. Since any
quotient of $G/O_\infty(G)$ is a quotient of $G$, the group
$G/O_\infty(G)$ satisfies the hypothesis of the theorem. Therefore,
we may assume that $O_\infty(G)=1$. Let $\Soc(G)$ denote the socle
of $G$. In other words, $\Soc(G)$ is the direct product of minimal
normal subgroups of $G$, each of which is a direct product of
nonabelian simple groups. The group $G$ now is embedded into
$\Aut(\Soc(G))$. Therefore, it suffices to show that $|\Soc(G)|$ is
$m$-bounded. This is done by the following two steps.

\medskip

4) If $S$ is a direct factor of $\Soc(G)$, then the number of times
of $S$ appearing in $\Soc(G)$ is $m$-bounded.

\noindent\emph{Proof.} Suppose that the number of times of $S$ in
$\Soc(G)$ is $k$. Let $N$ be the direct product of $k$ copies of
$S$. Then $N$ is a normal subgroup of $G$. As in step~1,
$H=G/C_G(N)$ is embedded in $\Aut(N)=\Aut(S)\wr S_k$. Set $K:=H\cap
\Aut(S)^k$. Note that $N$ can be viewed as a subgroup of $H$ and
therefore $N$ is a subgroup of $K$. Also, $H/K$ is a permutation
group on $k$ letters. By step 1, there is an integer $n$ bounded in
terms of $m$ such that $H/K$ does not contain any alternating group
larger than $A_n$. Therefore, by Lemma~\ref{orbit}, the number of
orbits of $H/K$ on the power set of $k$ letters is at least
$a^{k/n}$ where $a>1$ is some constant. It follows that there exists
$l\in\{0,1,2,...,k\}$ such that the number of orbits of $H/K$ on the
subsets of cardinality $l$ is at least $a^{k/n}/(k+1)$.

Following Lemma~\ref{simple-extension}, let $C$ be a nontrivial
conjugacy class of size $s$ of $S$ invariant under $\Aut(S)$. Then
$C$ is also a conjugacy class of $\Aut(S)$. Consider ${k\choose l}$
conjugacy classes of $K$ which are products of $l$ copies of $C$ and
$k-l$ copies of the trivial conjugacy class. From the conclusion of
the previous paragraph, we obtain that there are at least
$a^{k/n}/(k+1)$ $H$-orbits of conjugacy classes of $K$ of size
$s^l$. Therefore, $H$ has at least $a^{k/n}/(k+1)d(k!)$ conjugacy
classes of the same size. By Lemma~\ref{factorial}, $k$ must be
bounded in terms of $m$.

\medskip

5) If $S$ is a direct factor of $\Soc(G)$, then $|S|$ is
$m$-bounded.

\noindent\emph{Proof.} Define $k, N, H$, and $K$ as in step 4. We
then have that $H/K$ is a permutation group on $k$ letters and
therefore $H/K$ is $m$-bounded. It follows that the largest
multiplicity of conjugacy class sizes of $K/N$ is bounded in terms
of $m$ since $H/N=G/NC_G(N)$ does not have more than $m$ conjugacy
classes of the same size. Note that $K/N$ is isomorphic to a
subgroup of $\Out(S)$, which is solvable by the Schreier's
conjecture. So $K/N$ is solvable and hence its order is $m$-bounded
by Theorem~\ref{Jaikin-solvable}. We then have that $|H:N|$ is
$m$-bounded.

Recall that the largest multiplicity of conjugacy class sizes of $H$
is not more than $m$. It follows that the largest multiplicity of
conjugacy class sizes of $N$ is bounded in terms of $m$ and
therefore the same thing happens to $S$. Using Theorem C, we obtain
that $|S|$ is $m$-bounded.
\end{proof}


\section*{Acknowledgement} I am grateful to Stefan Forcey for
various discussion leading to the proof of Lemma~\ref{alternating}.



\end{document}